\documentclass{article}
 
\usepackage{amsthm}
\usepackage{amsfonts}
\usepackage{amssymb}
\usepackage{verbatim}
\usepackage[all]{xy}
\usepackage{amsmath, amscd, amsthm}
\usepackage{longtable}
\usepackage{amscd}
\usepackage{mathrsfs}
\usepackage{graphicx}
\usepackage{latexsym}

\newtheorem{theorem}{Theorem}[section]
\newtheorem{lemma}[theorem]{Lemma}
\newtheorem{prop}[theorem]{Proposition}
\newtheorem{cor}[theorem]{Corollary}
\newtheorem{ex}[theorem]{Example}
\newtheorem{defn}[theorem]{Definition}
\newtheorem{remark}[theorem]{Remark}

\usepackage{latexsym}
\newcommand{\z}{{\mathbb Z}}
\newcommand{\gp}{\textrm{gp}}

\begin{document}

\title{Distortion in Free Nilpotent Groups}
\author{Tara Davis}

\maketitle

\begin{abstract}
We prove that a subgroup of a finitely generated free nilpotent group F is undistorted if and only if it is a retract of a subgroup of finite index in F. 
\end{abstract}

\let\thefootnote\relax\footnotetext{{\bf Keywords:} Subgroup Distortion; Free Nilpotent Group; Membership Problem.}
\let\thefootnote\relax\footnotetext{{\bf Mathematics Subject Classification 2000:} 20F65, 20F18, 20F05, 20F10.}

\section{Introduction}\label{intro}
\subsection{Background and Preliminaries}

The primary notion which will be investigated in this paper is that of distortion of a subgroup. This notion was first introduced by Gromov in \cite{gromov}.

\begin{defn} 
Let $H$ be a subgroup of a group $G$, where $H$ and $G$ are generated by the finite sets $S$ and $T$, respectively. Then the distortion function of $H$ in $G$ is defined as $\Delta^{G}_{H} : \mathbb{N} \rightarrow \mathbb{N} : n \mapsto \max \{ |w|_S : w \in H, |w|_T \leq n \}$, where $|w|_S$ denotes the word length of $w$ in $H$ with respect to the finite generating set $S$, and $|w|_T$ is defined similarly.
\end{defn}

We study distortion functions up to a natural equivalence relation in order to make this concept independent of the choice of finite generating sets $S$ and $T$. 

\begin{defn}\label{equiv}
We say that $f \preceq g$ if there exists $C>0$ such that $f(n) \leq Cg(Cn+C)$ for all $n \geq 0$. We say two functions are equivalent, written $f \approx g$, if $f \preceq g$ and $g \preceq f$. 
\end{defn}

Note that if $H$ is infinite, then it is always true that $\Delta_H^G(n) \succeq n$.

\begin{defn}
The subgroup $H$ of $G$ is said to be undistorted if $\Delta^{G}_{H}(n) \approx n$.
\end{defn}

If a subgroup $H$ is not undistorted, then it is said to be distorted, and it's distortion refers to the equivalence class of $\Delta_H^G(n)$. 

The following facts about distortion are well-known. If $[G:H] < \infty$ then $H$ is undistorted in $G$. This follows from the Reidemeister-Schreier rewriting procedure; the rewriting of an element in the finite set of generators of $H$ does not increase it's length. Moreover, if $H$ is a retraction of $G$, then $H$ must be undistorted. To see this, one takes a generating set of $H$ to be the images under the retract of a finite set of generators for $G$. Finally, if $M \leq H \leq G$ and both $M$ is undistorted in $H$ as well as $H$ is undistorted in $G$, then $M$ must also be undistorted in $G$; this follows from the definition of distortion. 

The distortion function measures the difference in the metrics induced by generators of $G$ and $H$. Intuitively, a subgroup $H$ of a group $G$ is highly distorted if one must travel a long distance in the Cayley graph of $H$ whereas traveling between the same points in $G$ takes a relatively short distance.

There has been a wide range of work done with regards to the study of distortion in finitely generated groups. For instance, the complete description of length functions on subgroups of finitely generated and finitely presented groups can be found in \cite{o} and \cite{os}. Other interesting finitely generated groups with fractional distortion are constructed in \cite{bridson}. Additionally, in \cite{osin}, the formula for distortion in finitely generated nilpotent groups and nilpotent Lie groups is obtained. 

We introduce the main object of our study as well as some of its basic properties.

\begin{defn}\label{fn}
A free $m$-generated class $c$ nilpotent group $G_{m,c}$ is a $c$-nilpotent group with generators $y_1, \dots y_m$ defined by the following universal property: given an arbitrary $d$-nilpotent group $H$ for $d \leq c$ and elements $h_1, \dots, h_m \in H$ there is a unique homomorphism $\phi : G_{m,c} \rightarrow H : y_i \mapsto h_i$ for all $i=1, \dots , m$.
\end{defn}

Note that free nilpotent groups are torsion-free. See, for example, \cite{baumslag}. We will provide a simple and concrete example of free nilpotent groups and some of their distorted subgroups, after fixing some notation. 

\vspace{.1in}

\begin{remark}\label{1.5}
{\it We use the notation that the commutator $[x_1,x_2]=x_1^{-1}x_2^{-1}x_1x_2$ and inductively define higher commutators by $[x_1, \dots, x_i]=[x_1,[x_2, \dots, x_i]],$ for $i \geq 3.$ The descending central series of a group $G$ is defined inductively as: $\gamma_1(G)=G$ and $\gamma_i(G)=[G,\gamma_{i-1}(G)]$. With this notation we have that the free nilpotent group $G_{n,c}$ has presentation given by $R/\gamma_{c+1}(R)$ where $R$ is the absolutely free group of rank $m$.}
\end{remark} 

\begin{ex}\label{1.6}
The free $2$-generated, $2$-nilpotent group is isomorphic to the $3$-dimensional integral Heisenberg group $$\mathcal{H}^3=\gp \langle a, b, c | c=[a,b], [a,c]=[b,c]=1 \rangle.$$ It has cyclic subgroup $\langle c \rangle_{\infty}$ which is distorted and in fact it has quadratic distortion. To see why this is true, notice that the word $c^{n^2}$ has quadratic length in $\langle c \rangle$, but that in $\mathcal{H}^3$, we have $$c^{n^2} = [a,b]^{n^2} = [a^n,b^n]$$ which has at most linear length.
\end{ex}

\subsection{Statement of Main Results}
The main result of this note is the following. It will be proved in Section \ref{s3}.
 
\begin{theorem}\label{t1}
Let $F$ be a free $m$-generated, $c$-nilpotent group. A subgroup $H$ in $F$ is undistorted if and only if $H$ is a retract of a subgroup of finite index in $F$.
\end{theorem}

When the undistorted subgroup $H$ is normal in $F$ we may further refine our classification.

\begin{cor}\label{b2}
Let $H$ be a nontrivial normal subgroup of the free $m$-generated, $c$-nilpotent group $F$, and assume that $c \geq 2$. Then $H$ is undistorted if and only if $[F:H]< \infty$.
\end{cor}

\section{Facts on Nilpotent Groups}

We record several well known facts about nilpotent and free nilpotent groups which will be used in the proof of Theorem \ref{t1}. For instance, nilpotent groups possess special commutator identities.

\begin{lemma}\label{t2}
If $G$ is $c$-nilpotent, then the following identities hold:
$$(1)  \hspace{.1in}    [x_1, \dots, yz, \dots, x_c]=[x_1, \dots, y, \dots, x_c][x_1, \dots, z, \dots, x_c].$$
$$(2)  \hspace{.1in}   [x_1^{n_1}, \dots, x_c^{n_c}]=[x_1, \dots, x_c]^{n_1 \cdots n_c} \textrm{ for } n_1, \dots, n_c \in \mathbb{Z}.$$ 
\end{lemma}

In \cite{hall} special cases of these facts are discussed. The formulas in Lemma \ref{t2} can be easily obtained from these special cases. 

The following Lemma is useful when proving that subgroups of nilpotent groups are of finite index.

\begin{lemma}\label{t3}
If $G$ is a finitely generated nilpotent group and $H \leq G$, then if some positive power of each element of a set of generators of $G$ lies in $H$, then $[G:H] < \infty$ and a positive power of every element of $G$ lies in $H$. 
\end{lemma}

A proof of this fact can be found in \cite{baumslag}.

\begin{lemma}\label{t4}
If $G$ is any finitely generated nilpotent group, and $H  \leq  G$ then $[G:HG'] < \infty$ implies $[G:H]<\infty$.
\end{lemma}

In \cite{hall}, a special case of this Lemma is proved. The more general result of Lemma \ref{t4} follows by a simple argument.

The following result of Magnus will help us in proving Theorem \ref{t1}.

\begin{prop}\label{magnus}
Let $R$ be an absolutely free group. For $1 \ne x \in R$, let the weight of x, $w(x)=m$, be the first natural number such that $x \in \gamma_m(R)$ but $x \notin \gamma_{m+1}(R)$. Then for nontrivial elements $x_1$ and $x_2$ having respective weights $\lambda_1$ and $\lambda_2$, we have that the weight of $x=[x_1,x_2]$ equals $\lambda_1+\lambda_2$ if $\lambda_1 \ne \lambda_2$. Moreover, $w(x) > \lambda_1+\lambda_2$ if and only if the subgroup generated by $x_1$ and $x_2$ is also generated by some $\overline{x_1}, \overline{x_2}$ with weights $\lambda_1$ and $\lambda_1+\mu$, respectively, where $\mu>0$ and in this case, the weight of $x$ is $2 \lambda_1+\mu$.
\end{prop}

A proof of Proposition \ref{magnus} can be found in \cite{magnus}.

\begin{lemma}\label{t5}
If $c > 1$ and $F$ is free $c$-nilpotent, then the centralizer of an element $x_1 \notin F'$ is of the form $\gamma_c(F) \times \langle a \rangle$, where $a \notin F'$.
\end{lemma}

\begin{proof}
Let $R$ be an absolutely free group with the same number of generators as $F$. As mentioned in Remark \ref{1.5}, we have that $F=R/\gamma_{c+1}(R)$. An element $x_2$ is contained in the centralizer of $x_1$ in $F$, $C_F(x_1)$, if and only if $x=[x_1,x_2] =1$ in $F$ if and only if $x \in \gamma_{c+1}(R)$. That is, if considered as words in the absolutely free group $R$, $w(x) \geq c+1$. If $w(x_2) =1$ then by Proposition \ref{magnus}, and with notation as in Proposition \ref{magnus}, $w(x) \geq c+1$ which is equivalent to saying that $2+ \mu \geq c+1$; i.e. $1+\mu \geq c$. This means that $\gp \langle x_1, x_2 \rangle = \gp \langle \overline{x_1}, \overline{x_2} \rangle$ where $w(\overline{x_1})=1$ and $w(\overline{x_2})=1+\mu \geq c$, which occurs if $\overline{x_2} \in \gamma_c(R)$. Observe that if $w(x_2) \ne 1$ then by Proposition \ref{magnus}, $w(x)=w(x_2)+1 \geq c+1$ hence $w(x_2) \geq c$ which implies that $x_2 \in \gamma_c(R)$. Therefore, we have that $x_2 \in \gp \langle \overline{x_1} \rangle \times \gamma_c(R)$, with the understanding that in case $w(x_2) \ne 1$ we take $\overline{x_1}=x_1$ and $\overline{x_2}=x_2$. 

Hence, the image $x_2\gamma_{c+1}(R)$ in $F$ belongs to $\langle \overline{x_1}\gamma_{c+1}(R) \rangle \times (\gamma_c(R)/\gamma_{c+1}(R))$. The product is direct: the intersection is trivial because $c>1$ implies that $\langle \overline{x_1} \rangle \cap \gamma_c(R) \subseteq \langle \overline{x_1} \rangle \cap \gamma_2(R) = \{1\}$ because $w(x_1)=1$. Let $\langle y \gamma_{c+1}(R) \rangle$ be the unique maximal cyclic subgroup of the free nilpotent group $F$ containing $x_1 \gamma_{c+1}(R)$. This subgroup is the isolator of the cyclic subgroup. We will show that $$\langle y\gamma_{c+1}(R) \rangle \times (\gamma_c(R)/\gamma_{c+1}(R))$$ is the centralizer of $x_1$.  One inclusion has already been shown. It suffices to observe that $y \in C_F(x_1)$. This follows because there exists $n \in \z$ with $y^n \gamma_{c+1}(R) = x_1 \gamma_{c+1}(R)$. 
\end{proof}

\begin{prop}\label{t6}
Let $F$ be a free $m$-generated, $c$-nilpotent group with free generators $a_1, \dots a_m$, for $c \geq 1$. Suppose $b_1, \dots b_k \in F$ are such that $\{ b_1F', \dots b_kF' \}$ is a linearly independent set in the free abelian group $F/F'$ then $K:= gp \langle b_1, \dots b_k \rangle$ is free $c$-nilpotent.
\end{prop}

For a proof of Proposition \ref{t6} refer to \cite{neumann}.

The following result of Osin will be very useful to us later on.

\begin{prop}\label{osin}
Let $G$ be a finitely generated nilpotent group, and $H$ a subgroup of $G$. Let $H^0$ be the collection of elements of $H$ having infinite order. For $m \in H^0$, let the weight of $m$ in $G$ be defined by $$v_G(m) = \max \{ k | \langle m \rangle \cap G_k \ne \{ 1 \} \}$$ and similarly for $v_H(m)$.
Then $$\Delta^{G}_{H}(n) \approx n^{r}$$ where $$r = \displaystyle\max_{m \in H^0} \frac{v_G(m)}{v_H(m)}.$$ \end{prop}

A proof of this fact can be found in \cite{osin}.

\begin{cor}
If $G$ is nilpotent of class $c$ and $H$ is cyclic, then $$\Delta^{G}_{H}(n) \approx n^{d}$$ where $d \in \mathbb{N}$ and $d \leq c$. 
\end{cor}

\section{Undistorted Subgroups in Free Nilpotent Groups}\label{s3}
From this point on, all notation is fixed. Let $F$ be a free $m$-generated, $c$-nilpotent  group with free generators $a_1, \dots, a_m$, for $c \geq 1$. Suppose that $H$ is any nontrivial subgroup of $F$. Consider the group $HF'/F'$. Being a subgroup of the free abelian group $F/F'$, it is free abelian itself. Denote the free generators of $HF'/F'$ by $b_1F', \dots b_kF'$, where each $b_i \in H$, so $k=\textrm{rank}(HF'/F')$. Without loss of generality, $k>0$, for if $k=0$ then $H \subset F'$ so by Proposition \ref{osin}, $H$ is a distorted subgroup in $F$. We can assume further that $b_1, \dots b_k, a_{k+1}, \dots a_m$ are independent modulo $F'$. 

Let $D=\textrm{gp} \langle a_1, \dots a_k \rangle$. Consider the map $r : F \rightarrow D:$ 

  \[
     r(a_i) = 
     \begin{cases}
	 a_i & \mathrm{if \ } i \leq k, \\
	 1 &\mathrm{if \ } i>k.
     \end{cases}
  \]

Then $r$ is a retraction of $F$. This is clear: $r$ is a homomorphism because $F$ is free, and $r$ restricted to $D$ is the identity map. Let $N = \ker(r)$. 

\begin{lemma}\label{c2}
We have $[F:HN] < \infty$.
\end{lemma}

\begin{proof}
The elements $b_1, \dots , b_k, a_{k+1}, \dots, a_m$ generate a subgroup $S$ of finite index in $F$. This follows because the elements $b_1, \dots , b_k, a_{k+1}, \dots, a_m$ are linearly independent modulo $F'$, so $SF'/F'=\gp\langle b_1F', \dots, b_kF',a_{k+1}F', \dots, a_mF', F'\rangle$ is free abelian of rank $m$ and is a subgroup of $F/F'$. Therefore we have that $[F/F':SF'/F']< \infty$, which implies that $[F:SF']< \infty$. Hence by Lemma \ref{t4}, we have that $[F:S]<\infty$. Because $N$ is generated by $a_{k+1}, \dots, a_m$ and $H$ contains $b_1, \dots, b_k$, then $HN$ contains $S$, so $[F:HN] <\infty$.

\end{proof}

The following Lemmas are working towards proving that for $H$ undistorted, $H \cap N =\{1\}$, which would essentially complete the proof of Theorem \ref{t1}.

\begin{lemma}\label{c3}
If $H \cap N \ne \{1\}$ then $N \cap H \cap \gamma_c(F) \ne \{1\}$.
\end{lemma}

\begin{proof}
Observe that the Lemma is true in case $c=1$, so in the proof we assume that $c \geq 2$. Because $H$ is nilpotent group, and $H \cap N$ is nontrivial normal subgroup, we must have $Z(H) \cap H \cap N \ne \{1\}$. Observe that $$Z(H) = (\displaystyle\cap_{h \in H}C_F(h)) \cap H \leq C_F(b_1) \cap H$$ which by Lemma \ref{t5} has the form $(\gamma_c(F) \times \langle a \rangle) \cap H$ where $a \notin F'$. Now observe that $H \cap N \leq F'$. This follows because the image of $H$ in $F/F'$ is generated by $\{b_1, \dots, b_k \}$ and the image of $N$ in $F/F'$ is generated by $\{a_{k+1}, \dots, a_m\}$, so because the set $\{b_1, \dots, b_k, a_{k+1}, \dots, a_m \}$ is independent, the intersection $HF'/F' \cap NF'/F' =1$, so $(H \cap N)F'/F'=1$ which implies that $H \cap N \subset F'$. Thus we have $$Z(H) \cap N \leq (\gamma_c(F) \times \langle a \rangle) \cap F' \leq \gamma_c(F).$$ Therefore, there is a nontrivial element in $Z(H) \cap N \cap \gamma_c(F)$ as required.
\end{proof}

\begin{lemma}\label{c4}
If $N \cap H \cap \gamma_c(F) \ne \{1\}$ then $H$ is distorted.
\end{lemma}

\begin{proof}
Let $1 \ne u \in N \cap H \cap \gamma_c(F)$. We will show that that $\langle u \rangle \cap \gamma_c(H) = \{1\}$. For if $u^r \in \gamma_c(H)$ for some $0 \ne r \in \z$, then $u^r$ is a product of $c$-long commutators of the from $[y_1, \dots, y_c]^{\pm1}$ where $y_i$ is either one of $b_1, \dots, b_k$ or an element of $F'$ since $H$ is generated by $b_1, \dots, b_k$ and $F' \cap H$. But if one of the $y_i$'s belongs to $F'$, then the commutator is trivial because it is a $c+1$-long commutator in $F$. It follows that $u^r \in \gp \langle b_1, \dots, b_k \rangle \cap N$. 

By Lemma \ref{c2}, the subgroup $S=\gp\langle b_1, \dots, b_k, a_{k+1} \dots, a_m\rangle$ has finite index in $F$. This implies by Lemma \ref{t3} that $$[r(F):r(S)]=[D:\gp \langle r(b_1), \dots, r(b_k)\rangle]<\infty.$$ Therefore, we also have that $$[D/D':\gp \langle r(b_1), \dots, r(b_k)\rangle D'/D']<\infty$$ and so $\{r(b_1)D', \dots, r(b_k)D'\}$ is linearly independent in the free abelian group of rank $k$, $D/D'$. By Proposition \ref{t6} we have that both $$\gp \langle r(b_1), \dots, r(b_k) \rangle \textrm{ and } \gp \langle b_1, \dots, b_k \rangle$$ are free $k$-generated, $c$-nilpotent groups. This implies that the intersection $\gp \langle b_1, \dots, b_k \rangle \cap N$ is trivial, because $N=\ker(r)$. 

Hence $\langle u \rangle \cap \gamma_c(H) = \{1 \}$ and $1 \ne u \in \gamma_c(F)$. It follows by Propsotion \ref{osin} that the distortion of the cyclic subgroup $\langle u \rangle$ in $F$ is greater than its distortion in $H$. Thus $H$ cannot be undistorted in $F$. 
\end{proof}

\begin{cor}\label{c1}
If $H \cap N \ne \{ 1 \}$, then $H$ is distorted.
\end{cor}

\begin{proof}
This follows directly from Lemmas \ref{c3} and \ref{c4}.
\end{proof}

Now we proceed with the proof of Theorem \ref{t1}.  
\begin{proof}
As mentioned in Section \ref{intro}, every retract of a subgroup having finite index in any group $G$ is undistorted. Conversely, if $H$ is undistorted in $F$ then by Corollary \ref{c1} we have that $H \cap N = \{ 1 \}$. Then by Lemma \ref{c2}, $H$ is a retract of the subgroup $HN$ of finite index in $F$, as required.
\end{proof}

We also proceed with the proof of Corollary \ref{b2}.
\begin{proof}
We use the notation already established in this Section. Observe that if $k=m$ then we have by definition of $k$ that $[F:H]<\infty$. If by way of contradiction we suppose that $k<m$, then by Corollary \ref{c1}, $H$ undistorted implies that $H \cap N =\{1\}$. It follows by the normality of $H$ and $N$ and the fact that $b_1 \in H$ and $a_m \in N$ that $[b_1,a_m]=1$. On the other hand, by Proposition \ref{t6}, we have that $\gp \langle b_1, a_m \rangle$ is free nilpotent of class at least $2$, a contradiction.
\end{proof}

\section{Examples and Discussion}

\begin{ex}\label{4.1}
In the formulation of Theorem \ref{t1}, one may not replace ``retract of a subgroup of finite index" by ``finite index subgroup in a retract", although this is true in some cases (e.g. $\langle a^2 \rangle$ in $\mathcal{H}^3$).

\noindent For a counterexample, consider the cyclic subgroup $H = \langle a^2[a,b]^3 \rangle$ of the free $2$-generated, $2$-nilpotent group $F = \langle a, b | [a,[a,b]]=[b,[a,b]] =1 \rangle$. Since no non-trivial power of the generator of $H$ is in $F'$, it follows that $H \cap F' = \{1\}$. Therefore, by Proposition \ref{osin}, $H$ is undistorted in $F$. By Theorem \ref{t1}, we know that $H$ is a retract of a subgroup of finite index in $F$. Following the steps of the proof, we arrive at the subgroup $M= \langle a^2[a,b]^3, b \rangle$. 

\noindent However, it should be remarked that $H$ is not a subgroup of finite index in a retraction of $F$. First, observe that $H$ is not a retraction itself. For, if by way of contradiction there were such a homomorphism $\phi : F \rightarrow H$, then we have equations $\phi(a)=(a^2[a,b]^3)^n$ and $\phi(b)=(a^2[a,b]^3)^m$ as well as $a^2[a,b]^3 = \phi(a)^2[\phi(a),\phi(b)]^3$. But this set of equations has no solutions, even modulo $F'$. Next, observe that $H$ is not a proper subgroup of finite index in any $K \leq F$. This follows because $H$ is a maximal cyclic subgroup in a torsion-free nilpotent group.
\end{ex}

\begin{ex}\label{4.2}
Freeness is necessary for the formulation. For instance, consider the case of non-free $5$-dimensional Heisenberg group $F=\mathcal{H}^5$ defined by the presentation $$\langle x,y,u,v,z | [x,y]=[u,v]=z,[x,z]=[y,z]=[u,z]=[v,z]=1 \rangle.$$ Then by Lemma \ref{osin}, $H=\mathcal{H}^3$ is an undistorted subgroup of $F$. However, as we will show, $H$ is not a retract of any subgroup $K$ of finite index in $F$.  For if by way of contradiction, $[F:K] < \infty$ and $H$ is a retract of $K$, then we would have that the Dehn functions $f_H \preceq f_K \approx f_F$ which implies that $n^3 \preceq n^2$. These facts about Dehn functions are well known and the reader may see \cite{allcock} or \cite{os2} for more information about the Dehn function of $\mathcal{H}^5$ and \cite{gersten} for more information on the Dehn function of $\mathcal{H}^3$. 
\end{ex}

The following result is a direct implication of the proof of Theorem \ref{t1}.

\begin{cor} \label{t7}
Every undistorted subgroup $H$ of $F$ is ``almost a retract" in the following sense: there exists a normal subgroup $N \leq F$ such that $HN$ is of finite index in $F$ and $H \cap N = \{ 1 \}$.
\end{cor}

\begin{cor}
The undistorted subgroup $H$ of $F$ is virtually free $c-$nilpotent.
\end{cor}

\begin{proof}
With the notation of Section \ref{s3}, we have that $r(H)$ contains the free subgroup $K=\gp \langle r(b_1), \dots, r(b_k) \rangle$. Because $[D:K]< \infty$ and $K \leq r(H)$ it follows that $[r(H):K]< \infty$. Finally, because $H \cap N = \{1\}$ we have that $r(H) \cong H/(H \cap N) \cong H$. 
\end{proof}

\begin{ex}
There are undistorted subgroups of free nilpotent groups that are not free. For example, consider again the $3$-dimensional Heisenberg group $\mathcal{H}^3$ and its subgroup $H = \gp \langle a^2, b, [a,b] \rangle$. Then $H$ is undistorted because it is of finite index in $\mathcal{H}^3$. Moreover, $H$ is not free because $H' = \langle [a^2,b] \rangle$ and so $H/H'$ contains the nontrivial torsion element $[a,b]$. However, the group $H$ is virtually free, as it contains the free nilpotent subgroup $\langle a^2,b \rangle$ of finite index.
\end{ex}

\subsection*{Acknowledgements}
The author is grateful to Alexander Olshanskii for his support, encouragement and many valuable discussions.

\noindent Tara C. Davis\\
Department of Mathematics\\
Vanderbilt University\\
1326 Stevenson Center\\
Nashville, Tennessee 37240, United States of America\\
tara.c.davis@vanderbilt.edu\\

\end{document}